\documentclass[a4paper,11pt]{article}
\usepackage[pagewise]{lineno}
\usepackage{amsmath}
\usepackage{amssymb}
\usepackage{mathrsfs}
\usepackage{amsfonts}
\usepackage{amsthm}
\usepackage{pstricks, pst-node, pst-text, pst-3d,psfrag}
\usepackage{graphicx}
\usepackage{indentfirst}
\usepackage{enumerate}
\usepackage{subfigure}
\usepackage{cite}
\usepackage[colorlinks=true]{hyperref}

\theoremstyle{plain}
\newtheorem{thm}{Theorem}[section]
\newtheorem{prop}[thm]{Proposition}
\newtheorem{cor}[thm]{Corollary}
\newtheorem{lem}[thm]{Lemma}
\theoremstyle{definition}

\theoremstyle{remark}
\newtheorem{rmk}{\textnormal{{\bf Remark}}}[section]

\numberwithin{equation}{section}

\makeatletter 
\@addtoreset{equation}{section}
\makeatother  

\newcommand\norm[1]{\lVert#1\rVert}

\newcommand\eps{\varepsilon}
\newcommand\ep{\epsilon}
\newcommand\om{\omega}

\begin{document}

\title{Sharpened dynamics alternative and its $C^1$-robustness for strongly monotone discrete dynamical systems\thanks{Supported by NSF of China No.11825106, 11771414 and 11971232.}}

\setlength{\baselineskip}{16pt}

\author {
Yi Wang and Jinxiang Yao\thanks{Corresponding author: jxyao@mail.ustc.edu.cn (J. Yao).}
\\[2mm]
School of Mathematical Sciences\\
University of Science and Technology of China\\
Hefei, Anhui, 230026, P. R. China
}
\date{}
\maketitle

\begin{abstract}
For strongly monotone dynamical systems, the dynamics alternative for smooth discrete-time systems turns out to be a perfect analogy of the celebrated Hirsch's limit-set dichotomy for continuous-time semiflows. In this paper, we first present a sharpened dynamics alternative for $C^1$-smooth strongly monotone discrete-time dissipative system $\{F_0^n\}_{n\in \mathbb{N}}$ (with an attractor $A$), which concludes that there is a positive integer $m$ such that any orbit is either manifestly unstable; or asymptotic to a linearly stable cycle whose minimal period is bounded by $m$. Furthermore, we show the $C^1$-robustness of the sharpened dynamics alternative, that is, for any $C^1$-perturbed system $\{F_\ep^n\}_{n\in \mathbb{N}}$ ($F_\ep$ not necessarily monotone), any orbit initiated nearby $A$ will admit the sharpened dynamics alternative with the same $m$. The improved generic convergence to cycles for the $C^1$-system $\{F_0^n\}_{n\in \mathbb{N}}$, as well as for the perturbed system $\{F_\ep^n\}_{n\in \mathbb{N}}$, is thus obtained as by-products of the sharpened dynamics alternative and its $C^1$-robustness.
The results are applied to nonlocal $C^1$-perturbations of a time-periodic parabolic
equations and give typical convergence to periodic solutions whose minimal periods are uniformly bounded.

\par

\end{abstract}

\section{Introduction}

The theory of monotone dynamical systems grew out of the
series of groundwork of M. W. Hirsch (\cite{H84,H88a} and \cite{H82,H85,H88b,H90,H89,H91}) and Matano \cite{M79,M86}. Over three decades since its development, the theory and applications have undergone extensive investigations and continue to expand. Large quantities of mathematical models
of ordinary, functional and partial differential equations or difference equations can generate
monotone dynamical systems. We refer to \cite{HS05,P02,S95,S17,ShYi98,Chue02,JZ05,FWW19} (and references therein) for details.

For continuous strongly monotone semiflows, the central and signature result is the so called Hirsch's generic convergence theorem, concluding that generic precompact orbits approach a set of equilibria. For this purpose, Hirsch introduced a fundamental building block of the theory, \emph{the Limit-Set Dichotomy}, which asserts that for a strongly monotone semiflow with compact orbit closures,

\vskip 2mm
{\it If $x<y$, then either $\omega(x)\ll \omega(y)$ or $\omega(x)=\omega(y)\subset E.$}

\vskip 2mm

\noindent Here, $\omega(x)$ and $\omega(y)$ are the $\om$-limit set of $x$ and $y$, respectively; $E$ denotes the set of equilibria; and $\omega(x)\ll \omega(y)$ means that $u\ll v$ for any $u\in\omega(x)$ and $v\in\omega(y)$.
As a matter of fact, almost all of the important results (including the generic convergence theorem) in the theory of monotone semiflows follow from this deceptively simple result. Later on, motivated by earlier work of Pol\'{a}\v{c}ik \cite{P89}, Smith and Thieme \cite{ST91} improved the limit set dichotomy and generic convergence for $C^1$-smooth strongly monotone semiflows (see also in \cite{HS05}).

However, it is unfortunate that Hirsch's Limit-Set Dichotomy fails (see \cite{HS05} or \cite[Section 2, p.385-386]{HS05maps}) for strongly monotone discrete-time systems. Hence, there is no priori information
on the structure of limit sets of typical trajectories. This turns out to be a major significant difference between semiflows and discrete-time systems. Although certain weak Limit-Set Dichotomy was mentioned for strongly monotone discrete systems in the literatures (see e.g., \cite{Ta92,DH91}), it is still unknown whether those weak versions can be viewed as an effective substitute for Hirsch's Limit-Set Dichotomy.
As a consequence, for strongly monotone discrete-time systems (mappings), there is no result analogous to Hirsch's generic convergence theorem up to now unless certain smoothness assumption is imposed.

Pol\'{a}\v{c}ik and Tere\v{s}\v{c}\'{a}k \cite{PT92} first proved that the
{\it generic convergence to cycles} occurs provided that the mapping $F$ is of class $C^{1,\alpha}$ (i.e., $F$ is a $C^1$-map with a locally $\alpha$-H\"{o}lder derivative $DF$, $\alpha\in (0,1]$). Here, a cycle means a periodic orbit of $F$. In \cite{PT92}, a critical insight for the inherent structure of strongly monotone discrete-time systems is the following statement (we now call it as {\it dynamics alternative}), which states (see \cite[Theorem 4.1]{PT92}):
For any $x$ with a relatively compact orbit, {\it either

{\rm (a)} the $\omega$-limit set $\omega(x)$ is a linearly stable cycle; or,

{\rm (b)} there exists $\delta>0$ such that for any $y\in X$ satisfying $y < x$ or $y > x$,  \begin{equation}\label{E:unst}
{\limsup_{n \to +\infty}||{F}^{n}x-{F}^{n}y||\geq\delta}.
\end{equation}}
\noindent Here, a cycle is called linearly stable if the spectral radius of the derivative $DF^p$ along the cycle (of minimal period $p$) is no more than 1 (see Section \ref{S:Not-Pre}). Pol\'{a}\v{c}ik and Tere\v{s}\v{c}\'{a}k \cite{PT92} discovered the alternative by classifying such scenario in terms of the sign of the principal Lyapunov exponent, as well as the exponential separation along $\omega(x)$ (see \cite{PT93,Mi91,JS08}) and
the idea of construction of stable manifolds in the so called Pesin's Theory. So, the additional assumptions that $F$ is ${C}^{1,\alpha}$ and injective on $\omega(x)$ cannot be dropped in \cite{PT92}.

The dynamics alternative plays a very crucial role in the study of generic behavior for smooth strongly monotone discrete-time systems. It deserves to point out that in discrete-time systems it is actually {\it a perfect analog of Hirsch's Limit-Set Dichotomy}. In our previous work \cite{WY20-1}, the present authors proved the dynamics alternative for $C^1$-smooth discrete-time systems (i.e., $F\in C^1$) by improving Tere\v{s}\v{c}\'{a}k's extended exponential separation Theorem (\cite[Theorem 2.1]{T94}). The generic convergence to cycles for $C^1$-smooth systems (see \cite[Corollary 2.2]{WY20-1}) is thus obtained as a by-product of the dynamics alternative.

In the present paper, we will focus on the sharpened versions of dynamics alternative for $C^1$-smooth discrete systems (see Theorem A), as well as their robustness for $C^1$-perturbed systems (see  Theorem C). To be more precise, we formulate some standing hypotheses:
\vskip 2mm

\noindent \textbf{(H1)} $(X,C)$ is a strongly ordered Banach Space.
\vskip 1mm

\noindent \textbf{(H2)} $F_{0}:X \to X$ is a compact ${C}^{1}$-map, such that for any $x\in X$, the Fr\'echet derivative $DF_{0}(x)$

\quad\  is a strongly positive operator, i.e., $DF_{0}(x)v\gg0$ whenever $v>0$.
\vskip 2mm

We have the following sharpened dynamics alternative for $C^1$-smooth mapping $F_0$:

\newtheorem*{thma}{\textnormal{\textbf{Theorem A}}}
\begin{thma}[Sharpened $C^1$-dynamics alternative]
\emph{Assume that \textnormal{(H1)-(H2)} hold. Assume also $F_{0}$ is pointwise dissipative. Then there is an integer $m>0$ such that, for any $x\in X$, either}

\textnormal{(a)} \emph{$\omega(x,F_{0})$ is a linearly stable cycle of minimal period at most $m$; or,}

\textnormal{(b)} \emph{there is a constant $\delta>$ 0 such that, for any $y \in X$ satisfying $y < x$ or $y > x$,}
$$\mathop{\limsup}\limits_{n \to +\infty}\|F_{0}^{n}x-F_{0}^{n}y\|\geq\delta.$$
\end{thma}

Theorem A concludes that, there exists an integer $m>0$ such that any orbit is either manifestly unstable; or asymptotic to a linearly stable cycle whose minimal period is bounded by $m$.
An immediate consequence of Theorem A is the following {\it improved generic convergence to cycles} for {\it $C^1$-smooth} systems.

\newtheorem*{corB}{\textnormal{\textbf{Corollary B}}}
\begin{corB}[Improved generic convergence for $C^1$-systems]
\emph{Let all hypotheses in Theorem A hold. Then there is an integer $m>0$ such that the set
$$Q_{0}:=\{x\in X:\omega(x,F_{0})\text{ is a linearly stable cycle of minimal period at most }m\}$$
contains an open and dense subset of $X$.}
\end{corB}

Corollary B was first proved by Hess and Pol\'{a}\v{c}ik \cite[Corollary 4]{PH93} under the additional assumptions that $F_0$ is ${C}^{1,\alpha}$ and injective. Tere\v{s}\v{c}\'{a}k \cite{T94} first tackled the problem of the lower $C^1$-regularity and removed the injectivity of $F_0$. By some rather indirect arguments (see \cite[Propositions 1.3 and 3.1]{T94}), he obtained the generic convergence to cycles for $C^1$-smooth discrete-time systems. Unfortunately, Tere\v{s}\v{c}\'{a}k's Theorem has not yet been published.

Meanwhile, the improved generic convergence for $C^1$-system (Corollary B) was further informally announced in Tere\v{s}\v{c}\'{a}k \cite[p.2]{T94}. Since then, Corollary B was quite frequently stated, {\it but without proofs}, in many literatures (see, e.g. \cite[Section 4, p.387]{HS05maps},\cite[Section 5, p.97]{HS05} or \cite{H99,Mi94b,P02,S97}). Here, we obtain such improved generic convergence as a direct corollary of our sharpened $C^1$-dynamics alternative (Theorem A).

\vskip 3mm
Motivated by Theorem A, we further consider $C^1$-perturbations of the $C^1$-smooth mapping $F_{0}$, and obtain the sharpened dynamics alternative for the $C^1$-perturbed systems.
More precisely, we present an additional standing hypothesis:

\vskip 3mm
\noindent \textbf{(H3)} Let $J=[-\epsilon_{0},\epsilon_{0}]\subset\mathbb{R}$, and $F:J\times X\to X; (\epsilon,x)\mapsto F_{\epsilon}(x)$ is a compact $C^{1}$-map, i.e., $DF(\epsilon,x)$ continuously depends on $(\epsilon,x)\in J\times X$.
\vskip 3mm

The following theorem reveals that sharpened dynamics alternative of $F_0$ is robust under the $C^1$-perturbation.

\newtheorem*{thmc}{\textnormal{\textbf{Theorem C}}}
\begin{thmc}[$C^1$-robustness for sharpened dynamics alternative]
\emph{Assume that \textnormal{(H1)-(H3)} hold. Assume also $F_{0}$ is pointwise dissipative with an attractor $A$. Let $B_{1}\supset A$ be an open ball such that
\begin{equation}
\sup\{\|F_{\epsilon}x-F_{0}x\|+\|DF_{\epsilon}(x)-DF_{0}(x)\|: \epsilon\in J,\,x\in B_{1}\}
\end{equation}
sufficiently small. Then there exists a solid cone $C_{1}\subset{\rm Int}C$, an open bounded set $D_1$ (with $B_1\supset D_1\supset A$) and an integer $m>0$ such that, for each $x\in D_1$ and $|\epsilon|$ sufficiently small, either}

\textnormal{(a)} \emph{$\omega(x,F_{\epsilon})$ is a linearly stable cycle of minimal period at most $m$; or,}

\textnormal{(b)} \emph{there is a constant $\delta>$ 0 such that, for any $y \in D_1$ satisfying $y <_1 x$ or $y >_1 x$,}
$$\mathop{\limsup}\limits_{n \to +\infty}\|F_{\epsilon}^{n}x-F_{\epsilon}^{n}y\|\geq\delta.$$
\end{thmc}

Theorem C, as well as its stronger version (Theorem \ref{T:im-da-perturb}) will be proved in Section \ref{S:im-da-perturb}. As a by-product of the $C^1$-robustness for sharpened dynamics alternative of $F_0$, we can obtain the following improved generic convergence for $C^1$-perturbations of the discrete-time systems generated by $F_0$.

\newtheorem*{corD}{\textnormal{\textbf{Corollary D}}}
\begin{corD}[Improved generic convergence for $C^1$-perturbations]
\emph{Let all hypotheses in Theorem C hold. Then there exists an open bounded set $D_1\supset A$ and an integer $m>0$ such that, for any $|\epsilon|$ sufficiently small, the set
$$Q_{\epsilon}:=\{x\in D_1:\omega(x,F_{\epsilon})\text{ is a linearly stable cycle of minimal period at most }m\}$$
contains an open and dense subset of $D_1$.}
\end{corD}

Under the additional assumptions that $F_0$ is ${C}^{1,\alpha}$ and injective, Corollary D was first proved by Hess and Pol\'{a}\v{c}ik \cite[Corollary 5]{PH93} for monotone perturbations, that is, $F_{\epsilon}$ is monotone for each $\eps\in J$. Tere\v{s}\v{c}\'{a}k \cite{T94} considered the problem of generic convergence for perturbations of $C^1$-smooth mapping $F_{0}$. More precisely, without the assumptions on the injectivity of $F_0$ and the monotonicity of $F_{\epsilon}$, he \cite[Corollary 5.1]{T94} proved that, for any $|\epsilon|$ sufficiently small, the set
$$R_\eps:=\{x\in D_1:\omega(x,F_{\epsilon})\text{ is a linearly stable cycle}\}$$
contains a generic (open and dense) subset of $D_1$. So, Corollary D here improved Tere\v{s}\v{c}\'{a}k's results by showing that the set of minimal periods of linearly stable cycles contained in $D_1$ is bounded by $m$, which is actually a direct corollary of our Theorem C.
\vskip 3mm

The approach of the present paper is motivated by Hess and  Pol\'{a}\v{c}ik \cite[Section 4]{PH93} and our recent works \cite{WY20-1,WY21}. However, in our current framework, the lower $C^1$ (instead of $C^{1,\alpha}$)-regularity and loss of the one-to-one property of $F_0$, as well as the non-monotonicity of the perturbations $F_\ep$, make our approach far more delicate and difficult. In order to overcome such a series of difficulties, we combined with the ideas and techniques developed in Tere\v{s}\v{c}\'{a}k \cite{T94} and in our works \cite{WY20-1,WY21}.
Among others, the main novelty of our approach is to construct a bundle map $T$ (induced by certain iteration of the perturbation $F_\ep$) over the Cartesian square of some bounded neighborhood nearby the attractor $A$. By appealing to the extended exponential separation (see \cite[Proposition 3.2 or Theorem A]{WY20-1}) with the alternative cone $C_1$, as well as certain delicate equivalence estimates in terms of the principal Lyapunov exponents (see Proposition \ref{cla-equi}), we accomplish our approach by establishing the
$C^1$-Robustness for Sharpened Dynamics Alternative (i.e., Theorem C, and its stronger version Theorem \ref{T:im-da-perturb}). This enables us to obtain as by-products all other main results, including the Improved Generic Convergence for $C^1$-systems and their $C^1$-perturbations (i.e., Corollary B and Corollary D).

This paper is organized as follows. In Section \ref{S:Not-Pre}, we agree on some notations and provide relevant definitions and preliminary results. Besides, we further present several crucial propositions (see Propositions \ref{cla-equi}-\ref{P:bdd-per-Fqeps}), which turn out to be very important for our proof. In Section \ref{S:im-da-perturb}, we will prove the $C^1$-robustness for sharpened dynamics alternative (Theorem C, and its stronger version Theorem \ref{T:im-da-perturb}). In Section \ref{S:bdd-per-Fqeps}, we will prove the boundedness of stable periods for $C^1$-perturbed systems (i.e., Proposition \ref{P:bdd-per-Fqeps}). Other main results, including Theorem A, Corollary B and Corollary D will be proved in Section \ref{S:proof-ThAB-CorD}. Finally, in Section \ref{S:exam}, we will present an example of a nonlocal perturbation for time-periodic parabolic equations to illustrate our abstract results.

\section{Notations and Preliminary results}\label{S:Not-Pre}

Let $(X,\|\cdot\|)$ be a Banach space. A cone $C\subset X$, i.e., $C$ is a closed convex subset of $X$ such that $\lambda C\subset C$ for all $\lambda>0$ and $C\cap(-C)=\{0\}$. $C$ is said to be a solid cone, if ${\rm Int}C\neq\emptyset$. We call $(X,C)$ a strongly ordered Banach Space if $C$ is a solid cone. For $x,y\in X$, we write $x\leq y$ if $y-x\in C$, $x<y$ if $y-x\in C\backslash\{0\}$, $x\ll y$ if $y-x\in{\rm Int}C$. The reversed signs are used in the usual way. Given two subsets $A$ and $B$ of $X$, $A\leq B$ ($A<B$, $A\ll B$) means that $x\leq y$ ($x<y$, $x\ll y$) holds for each choice of $x\in A$ and $y\in B$. A subset $B\subset X$ is \emph{unordered} if it does not contain points $x,y$ such that $x<y$.

Denote by $X^{*}$ the dual space of $X$. The set $C^{*}=\{l\in X^{*}:l(v)\geq 0 \text{ for any }v\in C\}$ is called the dual cone of $C$. If ${\rm Int}C\neq\emptyset$, then $C^{*}$ is indeed a cone in $X^{*}$ (see \cite{D85}). Let $C_{s}^{*}=\{l\in C^{*}:l(v)>0,\text{ for any }v\in C\backslash\{0\}\}$. Choose $e\in{\rm Int}C$  and let $\|x\| _{e}=\inf\{\rho>0:x\in\rho[ -e,e]\}$. We call $\|\cdot\|_{e}$ an order norm on $X$. In general, $\|\cdot\|$ is stronger than $\|\cdot\|_{e}$, i.e., there is a constant $r>0$ such that $\|x\|_{e}\leq r\|x\|$ for any $x\in X$ (see \cite{D85}).  A mapping $h:X\to X$ is called \emph{monotone} (\emph{strongly monotone}), if $x\leq y$ ($x<y$) implies $hx\leq hy$ ($hx\ll hy$).

In this paper, we sometimes also need to deal with arguments for another solid cone $C_1(\subset C)$. Therefore, for the sake of no confusion, we write $\leq_1, <_1, \ll_1$ as the corresponding order relation induced by the cone $C_1$ throughout the paper.

\vskip 2mm
For a continuous map $h:X\to X$, the orbit of $x\in X$ is $O(x,h)$=$\{h^{n}x:n\geq0\} $. The $\omega$-limit set of $x\in X$ is $\omega(x,h)=\mathop{\bigcap}\limits_{k\geq0}\overline{\{h^{n}x:n\geq k\}}$. Let $D$ be a subset of $X$, the orbit of $D$ is $O(D,h)=\mathop{\bigcup}\limits_{x\in D}O(x,h)$.
A point $x\in X$ is a \emph{periodic point of $h$}, if $h^{p}x=x$ for some integer $p\geq1$. $p$ is then a \emph{period} of $x$. Moreover, if $h^{l}x\neq x$ for $l=1,2,\cdots,p-1$, we call $x$ $p$-periodic. $p$ is the \emph{minimal period} of $x$. In particular, if $p=1$, we say $x$ is a fixed point of $h$. A set $K$ is called a \emph{cycle} if $K=O(x,h)$ for some periodic point $x$. For a $C^1$-smooth map $h$, $x\in X$ and $v\in X$, we define
$$\lambda(x,v,h)={\limsup_{n \to +\infty}\frac{\log\|DF^{n}(x)v\|}{n}}\quad \text{and} \quad\lambda_{1}(x,h)=\sup_{\substack{v\in X\\v\neq0}}\lambda(x,v,h).$$
We call $\lambda_{1}(x,h)$ the principal \emph{Lyapunov exponent of x} (with respect to $h$). A cycle $K=O(x,h)$ is \emph{linearly stable} if the spectral radius of the derivative $Dh^p$ along the cycle (of minimal period $p$) is no more than 1 (we also call $x$ as a linearly stable $p$-periodic point of $h$). In particular, if $p=1$, we say $x$ is a linearly stable fixed point of $h$. Let $B\subset X$. We say that $k$ is a \emph{stable period} for the restriction ${h|}_{B}$ if there is a linearly stable $k$-periodic point $x$ of $h$ such that the orbit $O(x,h)=\{h^{n}x:n=0,1,\cdots,k-1\}$ is contained in $B$. If $B=X$ we simply say that $k$ is a stable period of $h$. For brevity, we hereafter say $\omega(x,h)$ is a linearly stable cycle (of minimal period $p$), if $\omega(x,h)$ is a linearly stable cycle (of minimal period $p$) of $h$.

A continuous map $h:X\to X$ is called \emph{pointwise dissipative}, if there is a bounded subset $B\subset X$ such that $B$ attracts each point of $X$. An invariant set $A$ is said to be an \emph{attractor} of $h$, if $A$ is the maximal compact invariant set which attracts each bounded subset $B\subset X$. If $h:X\to X$ is compact and pointwise dissipative, then there is an attractor $A$ of $h$ (see e.g., \cite{Ha88}).

Fix any $M^{\ast}>0$. It is not difficult to see that one of the following alternatives must hold:
\begin{alignat*}{2}
\textnormal{(Alta)}: \emph{ there exists } z\in\omega(x,F_{0}) \emph{ such that } \|DF_{0}^{n}(z)\|<M^{\ast} \emph{ for any } n\geq1; \emph{or else},\quad\quad\quad\\
\textnormal{(Altb)}: \emph{ for any } z\in \omega(x,F_{0}), \emph{ there exists } n(z)\geq1 \emph{ such that } \|DF_{0}^{n(z)}(z)\|\geq M^{\ast}.\;\;\quad\quad\quad
\end{alignat*}

\begin{prop}\label{P:da}
{\rm (Dynamics alternative).} Assume that \textnormal{(H1)-(H2)} hold. Let $x\in X$ have a relatively compact orbit.
Fix $M^{\ast}>0$ sufficiently large, we have the following

\textnormal{(a)} If \textnormal{(Alta)} holds, then $\omega(x,F_{0})$ is a linearly stable cycle;

\textnormal{(b)} If \textnormal{(Altb)} holds, then there is a constant $\delta>$ 0 such that, for any $y \in X$ satisfying $y < x$ or $y > x$,$$\mathop{\limsup}\limits_{n \to +\infty}\|F_{0}^{n}x-F_{0}^{n}y\|\geq\delta.$$
\end{prop}

\begin{proof}
See Wang and Yao \cite[Theorem 2.1]{WY20-1}.
\end{proof}

\begin{rmk}\label{R:Altb}
In fact, if (Altb) holds, one may further obtain that there exists $w\gg0$ and a bounded integer-valued function $\nu(z)$ on $z\in \omega(x)$ such that $DF_{0}^{\nu(z)}(z)w\gg3w$ for any $z\in \omega(x)$ (see the claim in the proof of \cite[Proposition 4.2]{WY20-1}).
\end{rmk}

In the following, we will show that our classification of (Alta)-(Altb) is equivalent to a classification of (Altc)-(Altd), by Pol\'{a}\v{c}ik and Tere\v{s}\v{c}\'{a}k \cite{PT92}, in terms of the principal Lyapunov exponents as
\begin{alignat*}{2}
\quad\textnormal{(Altc)}: \emph{ there exists } z\in\omega(x,F_{0}) \emph{ such that } \lambda_{1}(z,F_{0})\leq0; \emph{or else},\ \ \quad\quad\quad\quad\quad\quad\quad\quad\quad\quad\quad\quad\\
\quad\textnormal{(Altd)}: \emph{ for any } z\in \omega(x,F_{0}), \emph{  } \lambda_{1}(z,F_{0})>0.\ \,\quad\quad\quad\quad\quad\quad\quad\quad\quad\quad\quad\quad\quad\quad\quad\quad\quad\quad\quad\quad\quad
\end{alignat*}

\noindent As we will see in the following sections, such equivalence is crucial in our approaches for the main results.

\begin{prop}\label{cla-equi}
Assume that \textnormal{(H1)-(H2)} hold. Let $M^{\ast}>0$ be large. Then
\textnormal{(Alta)} is equivalent to \textnormal{(Altc)}, and \textnormal{(Altb)} is equivalent to \textnormal{(Altd)}.
\end{prop}

\begin{proof}
It is clear that (Alta) implies that (Altc). So, it suffices to prove that (Altb) implies (Altd). For this purpose, we assume (Altb) and let $z_{n}:=F_{0}^{n}z$. By Remark \ref{R:Altb}, for any $z\in \omega(x,F_{0})$, we write $n_{1}:=\nu(z)$ and $n_{k+1}:=n_{k}+\nu(z_{n_{k}})$, $k\geq1$. The chain rule shows that
\begin{equation}\label{E-DFn0-w}
DF_{0}^{n_{k}}(z)w\gg 3^{k}w,\quad\quad\text{ for }k\geq1.
\end{equation}
Let $l=\mathop{\sup}\limits_{z\in \omega(x,F_{0})}\nu(z)$. Then
$$n_{k}=\nu(z)+\nu(z_{n_{1}})+\cdots+\nu(z_{n_{k-1}})\leq lk,$$
for any $k\geq1$. Hence
$$k\geq\frac{n_{k}}{l}.$$
Together with \eqref{E-DFn0-w}, this leads to
\begin{equation}\label{E-gam-DFn0-w}
(e^{-\gamma})^{n_{k}}DF_{0}^{n_{k}}(z)w\geq w,
\end{equation}
where $$\gamma:=\log(3^{\frac{1}{{\blue l}}})>0.$$
Recall that $\|\cdot\|$ is stronger than $\|\cdot\|_{w}$, i.e., there is a constant $r>0$ such that $\|x\|_{w}\leq r\|x\|$ for any $x\in X$. Therefore, \eqref{E-gam-DFn0-w} implies that $(e^{-\gamma})^{n_{k}}\|DF_{0}^{n_{k}}(z)w\|\geq\frac{1}{r}\|(e^{-\gamma})^{n_{k}}DF_{0}^{n_{k}}(z)w\|_{w}\geq\frac{1}{r}$. That is to say, $\|DF_{0}^{n_{k}}(z)w\|\geq\frac{1}{r}e^{\gamma n_{k}}$. As a consequence,
$$\lambda(z,w,F_{0})=\mathop{\limsup}\limits_{n \to +\infty}\frac{\log\|DF_{0}^{n}(z)w\|}{n}\geq\mathop{\limsup}\limits_{k \to +\infty}\frac{\log\|DF_{0}^{n_{k}}(z)w\|}{n_{k}}\geq\gamma>0.$$
Hence, $\lambda_{1}(z,F_{0})\geq\lambda(z,w,F_{0})\geq\gamma>0$, for any $z\in \omega(x,F_{0})$. Thus, we have proved that (Altb) implies (Altd), which completes the proof.
\end{proof}

\begin{prop}\label{P:per-even-mono}
Assume that \textnormal{(H1)-(H3)} hold and $F_0$ is pointwise dissipative with an attractor $A$. Let $B_1$ be an open ball containing $A$. If
$$\sup\{\|F_{\epsilon}x-F_0x\|+\|DF_{\ep}(x)-DF_0(x)\|:\ep\in J, x\in B_1\}<\ep^{\prime}$$ for some
$\ep^\prime >0$, then there is a solid cone $C_1\subset{\rm Int}C$, an open bounded set $D$ ($B_{1}\supset D\supset A$) and an integer $q>0$, such that

{\rm (i)}. $F_{\ep}^{n}(D)\subset D$ for any $\epsilon\in J$ and $n\geq q$; and

{\rm (ii)}. $F_{\ep}^{n}x\ll_1 F_{\epsilon}^{n}y$ whenever $\epsilon\in J$, $x<_1 y$ (with $x, y\in D$) and $n\geq q$.
\end{prop}

\begin{proof}
See Tere\v{s}\v{c}\'{a}k \cite[Theorem 5.1]{T94}.
\end{proof}

\begin{rmk}\label{Def-bundlemap}
For each $\epsilon\in J$ and $(x,y)\in D\times D$, we define the map
\begin{equation}\label{E-T-ep}
R_{\epsilon,(x,y)}:=\int_0^1{DF_{\epsilon}(sx+(1-s)y)ds},
\end{equation}
 and the iteration
\begin{equation}\label{E-iter-T-ep}
R_{\epsilon,(x,y)}^{(n)}=R_{\epsilon,(F_{\epsilon}^{n-1}x,F_{\epsilon}^{n-1}y)}\circ\cdots\circ R_{\epsilon,(F_{\epsilon}x,F_{\epsilon}y)}\circ R_{\epsilon,(x,y)},
\end{equation}
for $(x,y)\in D\times D$ and $n\ge 1$. By letting the integer $q>0$ in Proposition \ref{P:per-even-mono} larger, if necessary, it follows from \cite[Eq.(5.5) on p.17]{T94} that
\begin{equation}\label{E-fam-iter-T-ep}
\{R_{\epsilon,(x,y)}^{(n)}\in \mathcal{L}(X):(x,y)\in D\times D,\ n\geq q\}
\end{equation}
is a continuous family of compact linear operators which are {\it strongly positive with respect to the cone $C_1$}. In particular, we have
\begin{equation}\label{E-TF-rela}
F_{\epsilon}^nx-F_{\epsilon}^ny=R_{\epsilon,(x,y)}^{(n)}(x-y)
\end{equation}
and
\begin{equation}\label{E-TF-rela2}
DF_{\epsilon}^n(x)=R_{\epsilon,(x,x)}^{(n)},
\end{equation}
for any $x,y\in D$ and $n\ge q$.
\end{rmk}

\vskip 2mm
\begin{rmk}\label{R:Def-D1}
It deserves to point out that the cone $C_1$ is actually independent of all small $\ep^\prime$ (see \cite[Definition of $C_{1}$ on p.16]{T94}). In addition, for smaller $\ep^\prime$, one may always choose some open bounded subset $D_1$ (satisfying $D\supset\overline{D}_1\supset D_{1}\supset A$) such that, by letting the integer $q>0$ larger (if necessary), both $D$ and $D_1$ satisfies items (i)-(ii) in Proposition \ref{P:per-even-mono} (see \cite[Eq.(5.11) on p.19]{T94}).
\end{rmk}

\vskip 2mm
Throughout the paper, we always reserve the solid cone $C_1\subset{\rm Int}C$, the open bounded subset $D,D_1$ (with $D\supset\overline{D}_1\supset D_{1}\supset A$) and the integer $q>0$ as in Proposition \ref{P:per-even-mono} and
Remarks \ref{Def-bundlemap}-\ref{R:Def-D1}.

\begin{prop}\label{P:bdd-per-Fqeps}
Assume that \textnormal{(H1)-(H3)} hold and  $F_{0}$ is pointwise dissipative with an attractor $A$. Let $B_{1}\supset A$ be an open ball such that
\begin{equation}\label{E-C1-close}
\sup\{\|F_{\epsilon}x-F_{0}x\|+\|DF_{\epsilon}(x)-DF_{0}(x)\|:\epsilon\in J, x\in B_{1}\}
\end{equation}
sufficiently small. Then there exists an integer $m_1>0$ such that, for any $|\epsilon|$ sufficiently small, all the stable periods of ${F_{\epsilon}^{q}|}_{\overline{D}_1}$ are bounded above by $m_1$.
\end{prop}

Proposition \ref{P:bdd-per-Fqeps} will play a crucial role in our approach for the $C^1$-robustness for sharpened dynamics alternative in Section \ref{S:im-da-perturb} and Section \ref{S:proof-ThAB-CorD}. For the sake of clarity, we will postpone to Section \ref{S:bdd-per-Fqeps} the proof of Proposition \ref{P:bdd-per-Fqeps}.

\vskip 2mm
An immediate consequence of Proposition \ref{P:bdd-per-Fqeps} is the following

\begin{cor}\label{P:bdd-per-mono}
Let all hypotheses in Proposition \ref{P:bdd-per-Fqeps} hold. Then there exists an integer $m>0$ such that, for any $|\epsilon|$ sufficiently small, all the stable periods of ${F_{\epsilon}|}_{\overline{D}_{1}}$ are bounded above by $m$.
\end{cor}

\begin{proof}
By virtue of Proposition \ref{P:bdd-per-Fqeps}, we only need to show the following claim:
{\it If $z$ is a linearly stable $k$-periodic point of $F^q_\epsilon$, then $z$ is a linearly stable periodic point of $F_\epsilon$ of minimal period at most $kq$. }

For this purpose, we choose $j\mid kq$ be such that $F_{\epsilon}^{j}z=z$ and $F_{\epsilon}^{i}z\neq z$, for any $1\leq i\leq j-1$. Since $z$ is a linearly stable $k$-periodic point of $F_{\epsilon}^{q}$, one has
\begin{equation}\label{E-Fqeps-LS}
\mathop{\lim}\limits_{n\to+\infty}\|(DF_{\epsilon}^{kq}(F_{\epsilon}^{sq}z))^{n}\|^{\frac{1}{n}}
=\mathop{\lim}\limits_{n\to +\infty}\|DF_{\epsilon}^{nkq}(F_{\epsilon}^{sq}z)\|^{\frac{1}{n}}\leq1,
\end{equation}
for any $s=1,\cdots,k$. Recall that $j\mid kq$. Then, for each $i=0,1,\cdots,j-1$, there exist integers $s\in\{1,\cdots,k\}$ and $l\in\{0,\cdots,q-1\}$ such that $i=sq-l$. Let $d=kq/j$. Then
\begin{alignat*}{2}
\mathop{\lim}\limits_{n}\|(DF_{\epsilon}^{j}(F_{\epsilon}^{i}z))^{n}\|^{\frac{1}{n}}&\ =\mathop{\lim}\limits_{n}\|DF_{\epsilon}^{nj}(F_{\epsilon}^{i}z)\|^{\frac{1}{n}}
=\mathop{\lim}\limits_{n}\|DF_{\epsilon}^{nj}(F_{\epsilon}^{sq-l}z)\|^{\frac{1}{n}}\\
&\ =\mathop{\lim}\limits_{n}\|DF_{\epsilon}^{ndj}(F_{\epsilon}^{sq-l}z)\|^{\frac{1}{nd}}
=\mathop{\lim}\limits_{n}\|DF_{\epsilon}^{nkq}(F_{\epsilon}^{sq-l}z)\|^{\frac{1}{nd}}\\
&\ =\mathop{\lim}\limits_{n}\|DF_{\epsilon}^{kq-l}(F_{\epsilon}^{sq+(n-1)kq}z)\circ DF_{\epsilon}^{(n-1)kq}(F_{\epsilon}^{sq}z)\circ DF_{\epsilon}^{l}(F_{\epsilon}^{sq-l}z)\|^{\frac{1}{nd}}\\
&\ \leq\mathop{\lim}\limits_{n}(M_{1}^{2}\|DF_{\epsilon}^{(n-1)kq}(F_{\epsilon}^{sq}z)\|)^{\frac{1}{nd}}
=\mathop{\lim}\limits_{n}\|DF_{\epsilon}^{(n-1)kq}(F_{\epsilon}^{sq}z)
\|^{\frac{1}{n-1}\cdot\frac{n-1}{nd}}\\
&\overset{\eqref{E-Fqeps-LS}}{\leq}1,
\end{alignat*}
for any $i=0,\cdots,j-1$. Here, $M_{1}:=\max\{\|DF_{\epsilon}^{n}(F_{\epsilon}^{p}z)\|:0\leq n\leq kq,1\leq p\leq kq\}$. Therefore, $z$ is a linearly stable $j$-periodic point of $F_{\epsilon}$. Thus, we have proved the claim.

Let $m=m_1q$. Together with the claim, Proposition \ref{P:bdd-per-Fqeps} directly implies this Corollary.
\end{proof}

\begin{rmk}\label{R:bdd-per-mono}
Under the additional assumptions of the $C^{1,\alpha}$-regularity, the injectivity of $F_0$ and the monotonicity of the perturbations $F_\ep$, Hess and Pol\'{a}\v{c}ik \cite[Theorem 1 and Theorem 2]{PH93} succeeded in proving the boundedness of stable periods of ${F_{\epsilon}|}_{B}$, where $B\subset X$ is a bounded set. The approach in \cite{PH93} inherited the ideas and arguments in Pol\'{a}\v{c}ik and Tere\v{s}\v{c}\'{a}k \cite{PT92}, which is based on the exponential separation along $\omega(x)$ (see, e.g. Mierczy\'{n}ski \cite{Mi91}, Pol\'{a}\v{c}ik and Tere\v{s}\v{c}\'{a}k \cite{PT93}), as well as the idea of construction of stable manifolds in the so called Pesin's Theory (see \cite{Pesin76}). As a consequence, these additional assumptions cannot be dropped in \cite{PH93}. However, these additional assumptions are removed in our Proposition \ref{P:bdd-per-Fqeps} and Corollary \ref{P:bdd-per-mono}.
\end{rmk}

\section{$C^1$-robustness for sharpened dynamics alternative}\label{S:im-da-perturb}

In this section, we will focus on the $C^1$-robustness for sharpened dynamics alternative. As we mentioned in Section \ref{S:Not-Pre}, we hereafter always reserve the notations of the solid cone $C_1\subset{\rm Int}C$, the open bounded subset $D,D_1$ (with $D\supset\overline{D}_1\supset D_{1}\supset A$) and the integer $q>0$ as in Proposition \ref{P:per-even-mono} and
Remarks \ref{Def-bundlemap}-\ref{R:Def-D1}.

\vskip 2mm
Our main result of this section is the following
\begin{thm}\label{T:im-da-perturb}
{\rm ($C^1$-robustness for sharpened dynamics alternative).} Assume that \textnormal{(H1)-(H3)} hold. Assume also $F_{0}$ is pointwise dissipative with an attractor $A$. Let $B_{1}\supset A$ be an open ball such that
\begin{equation}
\sup\{\|F_{\epsilon}x-F_{0}x\|+\|DF_{\epsilon}(x)-DF_{0}(x)\|:\epsilon\in J, x\in B_{1}\}
\end{equation}
sufficiently small. Then there exists an integer $m_1>0$ such that, for each $x\in D_{1}$ and $|\epsilon|$ sufficiently small, either

\textnormal{(a)} $\omega(x,F_{\epsilon}^{q})$ is a linearly stable cycle of minimal period at most $m_1$; or,

\textnormal{(b)} there is a constant $\delta>$ 0 such that, for any $y \in D_{1}$ satisfying $y <_1 x$ or $y >_1 x$,
\begin{equation}\label{E-unstable}
\mathop{\limsup}\limits_{n \to +\infty}\|F_{\epsilon}^{nq}x-F_{\epsilon}^{nq}y\|\geq\delta.
\end{equation}
\end{thm}

\begin{rmk}\label{R:pf-ThA}
Theorem \ref{T:im-da-perturb} is a stronger version of Theroem C. In fact, on one hand, \eqref{E-unstable} clearly implies the item (b) of Theorem C. On the other hand, by the claim in the proof of Corollary \ref{P:bdd-per-mono}, Theorem \ref{T:im-da-perturb}(a) entails that  $\omega(x,F_{\epsilon})$ is a linearly stable cycle of minimal period at most $m=m_1q$, which is exactly the item (a) of Theorem C.
\end{rmk}

\vskip 2mm
\noindent
{\it Proof of Theorem \ref{T:im-da-perturb}}. For each $x\in D_1$ and $|\epsilon|$ sufficiently small, we will {\it first show that: Either
\textnormal{(a$^\prime$)} $\omega(x,F_{\epsilon}^{q})$ is a linearly stable cycle; or \textnormal{(b)} holds.}

To this purpose, we note that $\omega(x,F_{\epsilon}^{q})\subset\overline{D}_{1}\,(\subset D)$, since $F_{\epsilon}^{q}(D_{1})\subset D_{1}$. For simplicity of notation, we denote $F_{\epsilon}^{q}$ by $G_{\epsilon}$. Then $G_{\epsilon}$ induces a continuous  map on $D\times D$ as
\begin{equation}\label{E:defn-G1ep}
G_{1\epsilon}:D\times D\to D\times D;(x_{1},y_{1})\mapsto G_{1\epsilon}(x_{1},y_{1}):=(G_{\epsilon}x_{1},G_{\epsilon}y_{1}),
\end{equation}
for all $(x_{1},y_{1})\in D\times D$. We further define the bundle map $T=\{T_{(x_1,y_1)}:(x_1,y_1)\in D\times D\}$ along $D\times D$ as
\begin{equation}\label{E:defn-T}
T_{(x_{1},y_{1})}=R_{\epsilon,(x_{1},y_{1})}^{(q)},\,\,\text{ for }(x_{1},y_{1})\in D\times D,
\end{equation}
where $R_{\epsilon,(x_{1},y_{1})}^{(q)}$ is defined in \eqref{E-iter-T-ep} from Remark \ref{Def-bundlemap}. It then follows from \eqref{E-fam-iter-T-ep} in Remark \ref{Def-bundlemap} that
\begin{equation}\label{E:C1-positive}
\{T_{(x_1,y_1)}\in \mathcal{L}(X):(x_1,y_1)\in D\times D\}
\end{equation}
is a continuous family of compact linear operators that are {\it strongly positive with respect to $C_1$}.
By virtue of \eqref{E-TF-rela2}, we have
\begin{equation}\label{E:T-G-ep}
T_{(x_{1},x_{1})}=DG_{\epsilon}(x_{1}),\,\, \text{ for any }x_1\in D;
\end{equation}
and moreover,
\begin{equation}\label{E:T-G-ep-diff}
T_{(x_1,y_1)}^{(n)}(x_{1}-y_{1})=G_{\epsilon}^{n}x_{1}-G_{\epsilon}^{n}y_{1}
\end{equation}
for any $(x_1,y_1)\in D\times D$ and $n\ge 1$.
Here, $T_{(x_{1},y_{1})}^{(n)}=T_{G_{1\epsilon}^{n-1}(x_{1},y_{1})}\circ\cdots\circ T_{G_{1\epsilon}(x_{1},y_{1})}\circ T_{(x_{1},y_{1})}$.

\vskip 2mm
Let $$K=Cl(O(x,G_{\epsilon}))\times Cl(O(x,G_{\ep})),$$ where
$Cl(O(x,G_{\ep}))$ denotes the closure of $O(x,G_{\ep})$.
Clearly, $K\subset\overline{D}_{1}\times\overline{D}_{1}\subset D\times D$. Moreover, $K$ is compact, because $G_\ep$ is compact. We consider the bundle map $(G_{1\ep},T)$ restricted on $K\times X$. Then the extended exponential separation theorem for continuous maps (see Wang and Yao \cite[Proposition 3.2 or Theorem A]{WY20-1}) implies that, for each $(x_1,y_1)\in K$, one can find a unit vector $l_{(x_1,y_1)}\in {C_1}_s^*$ (continuously depending on $(x_{1},y_{1})$) and a subset of unit vectors $V_{(x_1,y_1)}\subset{\rm Int}C_1$ (with $\mathop{\cup}\limits_{(x_1,y_1)\in K}V_{(x_1,y_1)}$ being a compact subset of ${\rm Int}C_1$), such that the bundle map $(G_{1\ep},T)$ on $K\times X$ satisfies the following exponentially separated property: There exist constants $M>0$ and $0<\gamma<1$ such that
\begin{equation*}\label{3.2}
 \|T_{(x_1,y_1)}^{(n)}w\|\leq M\gamma^{n}\|T_{(x_1,y_1)}^{(n)}v\|,
\end{equation*}
 for all $(x_1,y_1)\in K, n\geq1, v\in V_{(x_1,y_1)}$ and $l_{(x_1,y_1)}(w)=0$ with $\|w\|=1$.
 \vskip 2mm

Now, fix $M^{\ast}>0$ (sufficiently large). We point out that, for the limit set $\omega(x,G_{\ep})$, one of the following two cases must occur: \vskip 1mm

Case (i): there exists $z\in\omega(x,G_{\ep})$ such that $\|DG_{\epsilon}^{n}(z)\|<M^{\ast}$  for any $n\ge 1$; or else, \vskip 1mm

Case (ii):  for any $z\in \omega(x,G_{\ep})$ there exists $n(z)\ge 1$ such that $\|DG_{\epsilon}^{n(z)}(z)\|\geq M^{\ast}$. \vskip 2mm

\noindent Together with the exponentially separated property of $(G_{1\ep},T)$ on $K\times X$, we can repeat the same arguments in Wang and Yao \cite{WY20-1} to obtain that: If case (i) occurs, then $\omega(x,G_{\eps})$ is a linearly stable cycle (see \cite[Proposition 4.1]{WY20-1}). In other words, case (i) implies \textnormal{(a$^\prime$)}.

\vskip 3mm
While, if (ii) holds, then there exists $w\gg_1 0$ and a bounded integer-valued function $\nu(z)$ on $z\in \omega(x,G_{\epsilon})$ such that
\begin{equation}\label{E:DG-ep-3w}
DG_{\epsilon}^{\nu(z)}(z)w\gg_1 3w,\ \text{for any }z\in \omega(x,G_{\epsilon})
\end{equation}
(see the claim in \cite[Proposition 4.2, on p.9812]{WY20-1}).
In such circumstance, we will prove \eqref{E-unstable} by contradiction.

Suppose that \eqref{E-unstable} does not hold. Then, for any $\delta>0$, there exists $y\in D_{1}$ satisfying $y<_1 x$ or $y>_1 x$ and $\|G_{\ep}^{n}y-G_{\ep}^{n}x\|\leq\delta$ for all $n$ sufficiently large. For simplicity,
we denote $x_{k}=G_{\ep}^{k}x$ and $y_{k}=G_{\ep}^{k}y$, for $k=0,1,\cdots.$
Let
$\tilde{m}=\mathop{\sup}\limits_{z\in \omega(x,G_{\ep})}\nu(z)$
and $H=\overline{G_{\ep}(D_1)}\subset\overline{D}_1\subset D$.
Clearly, $H$ is compact.
Recall that $T_{(x',y')}$ continuously depends on $(x',y')\in D\times D$. Then there is a small $\delta>0$ such that
\begin{equation}\label{E:T-w--w-w}
(T_{(x_{1}^{\prime},y_{1}^{\prime})}^{(\nu)}-T_{(x_{2}^{\prime},y_{2}^{\prime})}^{(\nu)})w\in \{v\in X:-w\ll_1 v\ll_1 w\}
\end{equation}
for any $\nu=1,2,\cdots,\tilde{m}$
and $(x_{1}^{\prime},y_{1}^{\prime}),(x_{2}^{\prime},y_{2}^{\prime})\in H$ with $\|x_{1}^{\prime}-x_{2}^{\prime}\|<2\delta$ and $\|y_{1}^{\prime}-y_{2}^{\prime}\|<2\delta$.
So, for such $\delta$, we suppose without loss of generality that there exist some $y>_1 x$ and an integer $N>0$ such that $\|y_{n}-x_{n}\|<\delta$ for any $n\geq N$. For each $n$, define \begin{equation}\label{E:def-xi-n}
\xi_{n}=\sup\{\xi>0:x_{n}+\xi w\leq_1 y_{n}\}.
\end{equation}
Then $\xi_{n}=\xi_{n}\|w\|_{w}=\|\xi_{n}w\|_{w}\leq\|y_{n}-x_{n}\|_{w}\leq r\|y_{n}-x_{n}\|\leq r\delta$, for any $n\geq N$.

Let $z_{n}\in\omega(x,G_{\epsilon})$ be such that $\|z_{n}-x_{n}\|\to0$ as $n\to\infty$. Clearly, $x_{n},y_{n},z_{n}\in H$, for any $n\ge 1$. Moreover, one can find an integer $N_1\ge 1$ such that $\|z_n-x_n\|<\delta$ for any $n\ge N_1$. Hence,
$\|z_n-y_n\|<2\delta$, for any $n\ge N_2\triangleq\max\{N,N_{1}\}$. Let
$\xi=\mathop{\sup}\limits_{n\ge N_{2}}\{\xi_n\}\leq r\delta$. Now, choose $l\geq N_2$ such that $\xi_{l}>\frac{1}{2}\xi$, and we have
\begin{alignat*}{2}
y_{l+\nu(z_{l})}-x_{l+\nu(z_{l})}&\ \ \ \overset{\eqref{E:T-G-ep-diff}}=T_{(y_l,x_l)}^{(\nu(z_{l}))}(y_l-x_l)\\
&\overset{\eqref{E:C1-positive}+\eqref{E:def-xi-n}}{\ge_1}\xi_l\cdot T_{(y_l,x_l)}^{(\nu(z_l))}w\\
&\quad\ =\xi_l\cdot(T_{(y_{l},x_{l})}^{(\nu(z_{l}))}-T_{(z_{l},z_{l})}^{(\nu(z_{l}))})w+\xi_l\cdot T_{(z_{l},z_{l})}^{(\nu(z_{l}))}w\\
&\quad\overset{\eqref{E:T-G-ep}}{=}\xi_l\cdot(T_{(y_{l},x_{l})}^{(\nu(z_{l}))}-T_{(z_{l},z_{l})}^{(\nu(z_{l}))})w
+\xi_l\cdot DG_{\epsilon}^{\nu(z_{l})}(z_{l})w\\
&\overset{\eqref{E:DG-ep-3w}+\eqref{E:T-w--w-w}}{\geq_1}-\xi_lw+3\xi_{l}w=2\xi_{l}w>_1\xi w.
\end{alignat*}
This entails that $\xi_{l+\nu(z_{l})}>\xi$, contradicting the definition of $\xi$.
Thus, we have obtained \eqref{E-unstable}. In other words, case (ii) implies (b).

\vskip 2mm
Finally, we will show (a), that is, the upper bound of the stable period of cycles. Actually, this is directly from Proposition \ref{P:bdd-per-Fqeps}, which entails that there exists an integer $m_1>0$ such that, all the stable periods of ${F_{\epsilon}^{q}|}_{\overline{D}_1}$ are bounded above by $m_1$.
Thus, we have completed the proof. \hfill $\square$
\vskip 3mm


\section{Boundedness of stable periods for $C^1$-perturbed Systems}\label{S:bdd-per-Fqeps}
We will focus on the proof of Proposition \ref{P:bdd-per-Fqeps} in this section. Throughout this section, we always assume that (H1)-(H3) hold and $F_0$ is pointwise dissipative
with an attractor $A$.

As we mentioned in the end of Section \ref{S:Not-Pre} (see Remark \ref{R:bdd-per-mono}), Hess and Pol\'{a}\v{c}ik \cite[Theorem 2]{PH93} has actually obtained Proposition \ref{P:bdd-per-Fqeps} (see also Corollary \ref{P:bdd-per-mono}) under the additional assumptions that
\vskip 1mm

(I) $F_0$ is of $C^{1,\alpha}$ ($C^1$ with locally $\alpha$-H\"{o}lder continuous derivative), and it is one-to-one;

(II) $F_{\epsilon}:X\to X$ is a monotone mapping for each $\ep\in J$.
\vskip 1mm

\noindent Since their proofs in \cite{PH93} inherited the arguments and techniques from Pol\'{a}\v{c}ik and Tere\v{s}\v{c}\'{a}k \cite{PT92} (mainly based on the exponential separation and idea of the so called Pesin's Theory), these two additional assumptions cannot be dropped in \cite{PH93}.
In this section, we will remove the additional assumptions (I)-(II) and prove Proposition \ref{P:bdd-per-Fqeps}.
\vskip 2mm

Before proceeding our proof, we reserve the solid cone $C_1(\subset{\rm Int}C)$, the open bounded subset $D,D_1$ (with $D\supset\overline{D}_1\supset D_{1}\supset A$) and the integer $q>0$ be defined in Proposition \ref{P:per-even-mono} and
Remarks \ref{Def-bundlemap}-\ref{R:Def-D1}.

As in Section \ref{S:im-da-perturb}, we define a family of mappings $G=\{G_\ep\}_{\ep\in J}$ as
$$G:J\times X\to X; (\epsilon,x)\mapsto G_{\epsilon}(x)\triangleq F_{\epsilon}^{q}x.$$
Clearly, $G$ also satisfies (H1)-(H3). Moreover, for each $\ep\in J$, Proposition \ref{P:per-even-mono} directly implies that $G_\ep:D\to D$ is {\it strongly monotone with respect to $C_{1}$}.

In order to prove Proposition \ref{P:bdd-per-Fqeps}, motivated by the approaches in Hess and Pol\'{a}\v{c}ik \cite{PH93}, we need several technical lemmas, in which
we will overcome a series of difficulties due to the lack of the assumptions (I)-(II). We will summarize our ideas in the ending remark of this Section (see Remark \ref{R:lack-hol-mono}).

\begin{lem}\label{L:conti-Lya}
Let $\Gamma\subset D$ be a compact set invariant under $G_{0}$. Suppose that $\lambda_{1}(z,G_{0})>0$ for each $z\in \Gamma$. Then there exists a $\delta_{0}>0$ and a neighbourhood $V$ ($\subset D$) of $\Gamma$ such that, for $\epsilon\in[-\delta_{0},\delta_{0}]$ and $y\in V$ with $O(y,G_{\epsilon})\subset V$, we have $\lambda_{1}(y,G_{\epsilon})>0$.
\end{lem}
\begin{proof}
Fix $M^{\ast}>0$ sufficiently large. It follows from Proposition \ref{cla-equi} that \textnormal{(Altd)} is equivalent to \textnormal{(Altb)}. So, if $\lambda_{1}(z,G_{0})>0$ for each $z\in \Gamma$, then \textnormal{(Altb)} holds for $G_0$:
\begin{equation}\label{E:=alb-G0}
\text{For any } z\in \Gamma, \text{ there exists }\,n(z)\ge 1 \text{ such that } \|DG_{0}^{n(z)}(z)\|\ge M^{\ast}.
\end{equation}
Let
$(G_{10},T)$ be the bundle map on $(D\times D)\times X$ defined in \eqref{E:defn-G1ep}-\eqref{E:defn-T} with $\eps=0$. Similarly as our arguments before \eqref{E:DG-ep-3w}, we obtain that $(G_{10},T)$ admits the extended exponentially separated property on $(\Gamma\times\Gamma)\times X$ with respect to the cone $C_1$. Then, together with \eqref{E:=alb-G0}, this implies that \eqref{E:DG-ep-3w} holds for $G_0$, that is,
there exists $w\gg_10$ and a bounded integer-valued function $\nu(z)$ on $z\in \Gamma$ such that \begin{equation*}\label{E:DG-0-3w}
DG_{0}^{\nu(z)}(z)w\gg_1 3w,\ \text{for any }z\in \Gamma.
\end{equation*}
Recall that $(\ep,x)\to G_\ep(x)$ is $C^1$. Then there exist a $\delta_0>0$, a neighbourhood $V$ ($\subset D$) of $\Gamma$ and a bounded integer-valued function $\nu(y)$ on $y\in V$ such that
\begin{equation}\label{E:DG-ep-3w-V}
DG_{\ep}^{\nu(y)}(y)w\gg_1 3w,\quad \text{for any }y\in V\,\text{ and } \ep\in[-\delta_0,\delta_0].
\end{equation}
Note also that
\begin{equation*}\label{E:DG-ep-posi}
DG_{\ep}^{\nu(y)}(y)=DF_{\ep}^{q\nu(y)}(y)\overset{\eqref{E-TF-rela2}}{=}R_{\epsilon,(y,y)}^{(q\nu(y))}.
\end{equation*}
It then follows from \eqref{E-fam-iter-T-ep} in Remark \ref{Def-bundlemap} that, for any $y\in V$, $DG_{\ep}^{\nu(y)}(y)$ is strongly positive with respect to $C_1$. So, together with \eqref{E:DG-ep-3w-V}, we can repeat the exactly same argument in Hess and Pol\'{a}\v{c}ik \cite[p.1318]{PH93} (via replacing the cone $C$ by the cone $C_1$) to obtain the conclusion of this lemma. We omit it here.
\end{proof}

Let $\{z_n\}$ be a sequence of linearly stable periodic points of ${G_{\epsilon_{n}}|}_{\overline{D}_{1}}$ with $\epsilon_{n}\to0$. To minimize the number of indices, we use the notation $G_{n}=G_{\epsilon_{n}}$. The ``limit set'' $\Lambda$ is defined by
\begin{equation}\label{E:defn-Lambda}
\Lambda:=\mathop{\bigcap}\limits_{j\ge 1}\overline{\mathop{\bigcup}\limits_{n\geq j}O(z_{n},G_{n})}.
\end{equation}
It is not difficult to see that (see \cite[Lemma 4.1]{PH93}) $\Lambda$ is a nonempty compact invariant (under $G_{0}$) subset of $\overline{D}_{1}$, and
$${\rm dist}(O(z_{n},G_{n}),\Lambda)\to 0,$$
 as $n\to+\infty$. Here, dist$(N,M):=\mathop{\sup}\limits_{a\in N}\mathop{\inf}\limits_{y\in M}\|a-y\|$.

\begin{lem}\label{L:exis-LS}
$\Lambda$ contains a linearly stable periodic point $z$ of $G_{0}$.
\end{lem}
\begin{proof}
Since $O(z_{n},G_{n})$ are linearly stable, it follows from Lemma \ref{L:conti-Lya} (by taking $\Gamma=\Lambda$) that there exists a $z^{\prime}\in\Lambda$ such that $\lambda_{1}(z^{\prime},G_{0})\leq0$. For such $z'\in \Lambda$, the invariance of $\Lambda$ implies that $\omega(z^{\prime},G_{0})\subset\Lambda$. Note that $\lambda_{1}(G_{0}^{n}z^{\prime},G_{0})=\lambda_{1}(z^{\prime},G_{0})\leq0$ for any $n\geq0$. Then, again by Lemma \ref{L:conti-Lya} (taking $\Gamma=\omega(z^{\prime},G_{0})$), one can find some $\tilde{z}\in\omega(z^{\prime},G_{0})$ such that $\lambda_{1}(\tilde{z},G_{0})\le 0$.

Recall that (H1)-(H2) hold for $G_0$. Then Proposition \ref{cla-equi} indicates that for $M^{\ast}>0$ large, there exists $z\in\omega(z^{\prime},G_{0})$ such that $\|DG_{0}^{n}(z)\|<M^{\ast}$ for any $n\geq1$. Hence, Proposition \ref{P:da}(a) implies that $\omega(z^{\prime},G_{0})$ is a linearly stable cycle. So, $\omega(z^{\prime},G_{0})=O(z,G_{0})$ and $z$ is a linearly stable periodic point of $G_{0}$.
\end{proof}

The following lemma provides a classification of the the orbits of $G_{\epsilon}$ nearby a linearly stable periodic point of $G_{0}$.

\begin{lem}\label{L:weak-da}
Let $z\in D$ be a linearly stable periodic point of $G_{0}$. Then, for any neighbourhood $V$ of $O(z,G_{0})$, there exist constants $\rho>0$ and $\delta_{1}>0$ such that for any $\epsilon\in[-\delta_{1},\delta_{1}]$ and $y\in X$ with $\|y-z\|<\rho$, one of the following alternatives must occur:

{\rm (i)}. $O(y,G_{\epsilon})\subset V$; \emph{or,}

{\rm (ii)}. There are positive integers $r, k$ such that $G_{\epsilon}^{r+k}y\gg_1 G_{\epsilon}^{r}y$ or $G_{\epsilon}^{r+k}y\ll_1 G_{\epsilon}^{r}y$.
\end{lem}

\begin{proof}
This lemma has been proved in \cite[Lemma 4.2]{PH93} under the $C^{1,\alpha}$-smooth assumption for $F_0$ (equivalently, for $G_0$). Here, we give our improved proof under the $C^1$-smooth assumption.

Without loss of generality, we assume that $z$ is a linearly stable fixed point of $G_{0}$. Let $V$ be any neighborhood of $z$, we will find $\rho>0$ and $\delta_1>0$ such that (i)-(ii) hold.

To this purpose, we denote $y_n=G_{\epsilon}^{n}y$ and let $u_n=y_{n}-z$, for $n\ge 0$. Then, $u_n$ satisfies the following iteration
\begin{equation}\label{E:un+1-un}
u_{n+1}=Au_{n}+g(u_{n})+H(\epsilon,u_{n}),\quad n\ge 0,
\end{equation}
where $$A=DG_{0}(z),$$ $$g(u)=\int_0^1{[DG_{0}(z+su)-DG_{0}(z)]uds},$$ and $H(\epsilon,u)=G_{\epsilon}(u+z)-G_{0}(u+z)$.

By \eqref{E-fam-iter-T-ep} and \eqref{E-TF-rela2} in Remark \ref{Def-bundlemap}, it is clear that $A=DG_{0}(z)\,(=R_{0,(z,z)}^{(q)})$ is a compact linear operator that are {\it strongly positive with respect to $C_1$}. Then, $A$ admits an invariant Krein-Rutman decomposition $X=X_{1}\bigoplus X_{2}$, that is,
there exists a unit vector $v\gg_1 0$ and constants $M\ge 1$, $\beta\in(0,1)$ such that  $X_{1}={\rm span}\{v\}$, $X_{2}\cap C_1=\{0\}$ with
\begin{equation}\label{E:KR}
\|A_{2}^{n}\|\leq M\beta^{n},\quad A_{2}:=A|_{X_{2}}.
\end{equation}
Let $P:X\to X_1$ be the natural projection onto $X_1$ along $X_2$, and let $Q=I-P$. Write $u^1=Pu$, $u^2=Qu$. We have the following property (see Sublemma in \cite[p.1322]{PH93} with proof on \cite[p.1324]{PH93}):
\vskip 2mm

\noindent {\bf (P1)} There are $C_{2}>0$ and $\alpha\in(0,1)$ such that for any $u\in X$ and $R>0$, if $\|u^2\|\leq C_{2}\|u^1\|$ and $\|u\|\geq R$, then either $u\gg_1B_{\alpha R}$ or $u\ll_1B_{\alpha R}$.
\vskip 2mm

\noindent Here, $B_{\alpha R}$ denotes the open ball of radius $\alpha R$ centered at $0$.
\vskip 2mm

Now, Choose a small number $C_3>0$ that
\begin{equation}\label{E:defn-C3}
(1+C_{2}^{-1})C_{3}M\|Q\|\sum\limits_{k=0}^{\infty}\beta^{k}<\frac{1}{3}.
\end{equation}
Recall that $G_{0}$ is $C^{1}$. Then for such $C_3>0$, there exists a $R>0$ so small that
\begin{equation}\label{E:defn-r}
\|g(u)\|\leq\int_0^1{\|DG_{0}(z+su)-DG_{0}(z)\|\cdot \|u\|ds}\leq C_{3}\|u\|,
\end{equation}
for any $\|u\|<R$. Moreover, by letting $R>0$ smaller (if necessary), we also have
\begin{equation}\label{E:Rsmall-V-r}
z+B_{R}\triangleq\{z+u:\norm{u}\le R\}\subset V,
\end{equation}
and there exists $\tilde{M}>0$ such that
\begin{equation*}\label{E:Rsmall-Gep-bound}
\norm{\frac{\partial G(\epsilon,z+u)}{\partial\epsilon}}<\tilde{M},\,\text{ for any } \epsilon\in[-\epsilon_{0},\epsilon_{0}]\ \text{and } \norm{u}\le R.
\end{equation*}
Then, we have $h(\epsilon)\triangleq\sup\limits_{\|u\|\leq R}H(\epsilon,u)\leq\epsilon\tilde{M}\to0$, as $\epsilon\to 0$. Choose $\delta_{1}>0$ such that
\begin{equation}\label{E:defn-delta0}
h(\epsilon)M\|Q\|\sum\limits_{k=0}^{\infty}\beta^{k}<\frac{\bar{R}}{3}
\end{equation}
for any $\epsilon\in[-\delta_{1},\delta_{1}]$, where $\bar{R}\triangleq(1+C_{2}^{-1})^{-1}R$.
Let
$$0<\rho<\min\{\alpha R, \left[3M\|Q\|(1+C_{2}^{-1})\right]^{-1}R\}\triangleq \min\{\alpha R, (3M\|Q\|)^{-1}\bar{R}\}.$$
Fix such $\rho>0$ and $\delta_1>0$, in order to prove (i)-(ii) of this lemma, it suffices to show the following {\it assertion: If $\epsilon\in[-\delta_{1},\delta_{1}]$ and $\|u_{0}\|<\rho$, then either}
\vskip 2mm
(i)$^{\prime}$\emph{ $\|u_{n}\|<R$ for any $n\geq0$; or,}

\vskip 2mm
(ii)$^{\prime}$\emph{ There exists an integer $n\geq1$ such that $\|u_{n}\|\geq R$ and $\|u_{n}^{2}\|\leq C_{2}\|u_{n}^{1}\|$.}
\vskip 2mm

\noindent Indeed, if (i)$^{\prime}$ holds, then \eqref{E:Rsmall-V-r} clearly yields that $y_{n}\in V$ (hence, (i) holds). On the other hand, if (ii)$^{\prime}$ holds, then property (P1) implies that $u_{n}\gg_1B_{\alpha R}$ or $u_{n}\ll_1B_{\alpha R}$. Noticing that $\|u_{0}\|<\rho<\alpha R$, one has $u_{n}\gg_1u_{0}$ or $u_{n}\ll_1u_{0}$ (hence, (ii) holds).
\vskip 2mm

Therefore, in the following, we will show the assertion. Suppose that (ii)$^{\prime}$ dose not hold. Then, one has $\|u_{n}\|<R$ or $\|u_{n}^{1}\|< C_{2}^{-1}\|u_{n}^{2}\|$, for all $n\geq1$. In such circumstance, we will prove (i)$^{\prime}$. To this purpose, one only needs to show
\begin{equation}\label{E:Q-dire-bdd}
\|u_{n}^{2}\|<\bar{R},\quad \text{ for any } n\ge 0.
\end{equation}
In fact, if $\|u_n^{2}\|<\bar{R}$ with $\|u_n^{1}\|<C_{2}^{-1}\|u_n^{2}\|$, then $\|u_n\|\leq\|u_n^{1}\|+\|u_n^{2}\|< R$. Thus, we have proved (i)$^{\prime}$.

\vskip 2mm
So, it remains to prove \eqref{E:Q-dire-bdd}. To this end, we will prove it by induction.
Clearly, $\|u_{0}^{2}\|\leq \|Q\|\cdot\|u_{0}\|<\rho\|Q\|<(3M)^{-1}\bar{R}<\bar{R}$, which means that
\eqref{E:Q-dire-bdd} holds for $n=0$. Suppose that \eqref{E:Q-dire-bdd} holds for $k=0,1,\cdots,n-1$. Return to the iteration \eqref{E:un+1-un}, we apply the projection $Q$ to \eqref{E:un+1-un} and obtain $$u_{n+1}^{2}=A_{2}u_{n}^{2}+Q[g(u_{n})+H(\epsilon,u_{n})].$$ Then the following ``variation-of-constants'' formula holds
\begin{equation}\label{E:var-of-cons}
u_{n}^{2}=A_{2}^{n}u_{0}^{2}+\sum\limits_{k=0}^{n-1}A_{2}^{n-k-1}Q[g(u_{k})+H(\epsilon,u_{k})].
\end{equation}
Since $\|u_{k}^{2}\|<\bar{R}$ for $k\le n-1$ and (ii)$^{\prime}$ dose not hold. Then one has $\|u_{k}\|<R$, for $k\le n-1$.
Therefore,
\begin{alignat*}{2}\label{4.8}
\|u_{n}^{2}\|&\quad\overset{\eqref{E:var-of-cons}+\eqref{E:KR}}{\leq} M\beta^{n}\|u_{0}^{2}\|+\sum\limits_{k=0}^{n-1}M\beta^{n-k-1}\|Q\|(\|g(u_{k})\|+\|H(\epsilon,u_{k})\|)\\
&\quad\,\overset{\eqref{E:defn-r}+\eqref{E:defn-delta0}}{\leq} M\|u_{0}^{2}\|+M\|Q\|C_3R\sum\limits_{k=0}^{\infty}\beta^{k}+\frac{\bar{R}}{3}\\
&\,\quad\quad\overset{\eqref{E:defn-C3}}{\leq} M\|u_{0}^{2}\|+\frac{2\bar{R}}{3}.
\end{alignat*}
Recall that $\|u_{0}^{2}\|<(3M)^{-1}\bar{R}$. Then we have $\|u_{n}^{2}\|<\bar{R}$. Thus, we have proved \eqref{E:Q-dire-bdd}, which completes the proof.
\end{proof}

By virtue of Lemma \ref{L:weak-da}, we have the following lemma.

\begin{lem}\label{P:localization}
Let $z$ be as in Lemma \ref{L:exis-LS}. Then there exists a subsequence $\{z_{n_j}\}$ such that $$\textnormal{dist}(O(z_{n_{j}},G_{n_{j}}),O(z,G_{0}))\to0.$$
\end{lem}
\begin{proof}
By the definition of $\Lambda$, there exists a subsequence $z_{n_{j}}$ and a sequence $m_{j}$ such that $\|G_{n_{j}}^{m_{j}}z_{n_{j}}-z\|\to0$. Hence, for any neighbourhood $V$ of $O(z,G_{0})$, there exists some $N>0$, such that $\|G_{n_{j}}^{m_{j}}z_{n_{j}}-z\|<\rho$ and $\epsilon_{n_{j}}\in[-\delta_{1},\delta_{1}]$ for any $j\geq N$ ($\rho$ and $\delta_{1}$ are from Lemma \ref{L:weak-da}).

Since $O(z_{n_{j}},G_{n_{j}})\subset\overline{D}_{1}\subset D$ and $G_{n_{j}}:D\to D$ is strongly monotone with respect to $C_{1}$,  one has $O(z_{n_{j}},G_{n_{j}})$ is unordered with respect to $<_1$ (see e.g., Tak\'{a}\v{c} \cite[Proposition 2.2]{Ta92}). This contradicts the statement in Lemma \ref{L:weak-da}(ii). Therefore, only Lemma \ref{L:weak-da}(i) holds; and hence, $O(z_{n_{j}},F_{n_{j}})\subset V$, for any $j\geq N$. We have completed the proof.
\end{proof}

The following lemma on ``local bifurcation" asserts that for $|\epsilon|$ sufficiently small, the stable periods of $G_{\epsilon}$ in the neighbourhood of a linearly stable $k$-periodic orbit of $G_{0}$, can not exceed $k$. Such lemma is straightforward adopted from \cite[Proposition 4.3]{PH93}.

\begin{lem}\label{P:local-birfu}
Let $z\in D$ be a linearly stable $k$-periodic point of $G_{0}$. Then there exists a neighborhood $V$ of $O(z,G_{0})$ and $\delta_{2}>0$ such that, for any $\epsilon\in(-\delta_{2},\delta_{2})$ all cycles of $G_{\epsilon}$ contained in $V$ have the minimal period $k$.
\end{lem}
\begin{proof}
See Hess and  Pol\'{a}\v{c}ik \cite[Proposition 4.3]{PH93}.
\end{proof}
\vskip 2mm

Now, we are ready to prove Proposition \ref{P:bdd-per-Fqeps}.
\vskip 2mm
\noindent
{\it Proof of Proposition \ref{P:bdd-per-Fqeps}.} We first show that all the stable periods of ${G_0|}_{\overline{D}_1}$ are bounded. Suppose not, there is a sequence $\{z_{n}\}$ of linearly stable periodic points of ${G_{0}|}_{\overline{D}_{1}}$ such that, the minimal period of $\{z_n\}$ tends to $+\infty$. Let $\Lambda$ be defined in \eqref{E:defn-Lambda}. Lemma \ref{L:exis-LS} entails that $\Lambda$ contains a linearly stable $k$-periodic point $z$ of $G_{0}$, for some integer $k\geq1$. Moreover, it follows from Lemma \ref{P:localization} that, there exists a subsequence $\{z_{n_{j}}\}$ such that
$$\textnormal{dist}(O(z_{n_{j}},G_{0}),O(z,G_{0}))\to 0.$$
Hence, by Lemma \ref{P:local-birfu} (with $\ep\equiv0$) there exists an integer $N>0$ such that, $z_{n_{j}}$ are all $k$-periodic point of ${G_{0}|}_{\overline{D}_{1}}$, for $j\geq N$, which contradicts the choice of $z_n$. So, all the stable periods of ${G_0|}_{\overline{D}_1}$ are bounded.

Next, fix an $m_1>0$ to be an upper bound of the stable periods of ${G_0|}_{\overline{D}_1}$.
we will show that the stable periods of ${G_{\epsilon}|}_{\overline{D}_{1}}$ are bounded above by $m_1$, for any $|\epsilon|$ sufficiently small. Suppose on the contrary that there exists a sequence $\{\hat{z}_n\}$ of linearly stable periodic points of ${G_{\epsilon_{n}}|}_{\overline{D}_{1}}$ with $\epsilon_{n}\to 0$, such that the minimal period of $\hat{z}_n$ is larger than $m_1$ for any $n\geq1$.
By repeating the exactly same arguments as in the paragraph above, we can obtain a subsequence $\{\hat{z}_{n_{j}}\}$ of $\{\hat{z}_n\}$ and a linearly stable cycle $O(\hat{z},G_0)$ of ${G_0|}_{\overline{D}_1}$
such that
$\textnormal{dist}(O(\hat{z}_{n_{j}},G_{\eps_{n_j}}),O(\hat{z},G_{0}))\to 0$ as $j\to \infty$.
Recall that the minimal period of $O(\hat{z},G_0)$ is bounded from above by $m_1$. Then, by Lemma \ref{P:local-birfu} again, one obtains that the minimal period of $O(\hat{z}_{n_{j}},G_{\eps_{n_j}})$ is bounded from above by $m_1$ for all $j$ sufficiently large, a contradiction. Thus, we have proved that all the stable periods of ${F_{\epsilon}^{q}|}_{\overline{D}_1}$ are bounded above by $m_1$, for any $|\epsilon|$ sufficiently small. The proof is completed.
\hfill $\square$

\vskip 3mm
\begin{rmk}\label{R:lack-hol-mono}
Our proof of Proposition \ref{P:bdd-per-Fqeps} is motivated by the approach in Hess and  Pol\'{a}\v{c}ik \cite[Section 4]{PH93}. However, in our framework, the lower $C^1$ (instead of $C^{1,\alpha}$)-regularity and lack of the one-to-one property of $F_0$, as well as the non-monotonicity of the perturbations $F_\ep$, make our approach more delicate and difficult. In order to overcome such a series of difficulties, we combined with the ideas and techniques developed in our previous works \cite{WY20-1,WY21} and in Tere\v{s}\v{c}\'{a}k \cite{T94}. Among others, the main novelty of our approach is to construct a bundle map $T$ over the Cartesian square $\Lambda\times \Lambda$ of the limit-set $\Lambda$ rather than $\Lambda$ itself. By utilizing the extended exponential separation (see \cite[Proposition 3.2 or Theorem A]{WY20-1}) on $\Lambda\times \Lambda$ with the alternative cone $C_1$, as well as certain delicate estimates (for example, \eqref{E:defn-C3}-\eqref{E:defn-r}), we accomplish our approach by proving the critical Lemmas \ref{L:conti-Lya}-\ref{L:weak-da}, which enables us to remove the additional assumptions (I)-(II) mentioned at the beginning of this section.
\end{rmk}

\vskip 3mm
Before ending this section, we specially point out that, motivated by the arguments in the second paragraph of the proof for Proposition \ref{P:bdd-per-Fqeps}, one may even obtain the following fact: Let
$m_{\ep,\overline{D}_1}\triangleq\sup\left\{k>0:\text{there exists } x\in \overline{D}_1\,\text{ such that }O(x,F_\ep)\,\text{is a linearly stable } k{\text{-cycle} }\right\}.$ Then
\begin{itemize}
\item {\it If $m_{0,\overline{D}_1}\leq m$ for some integer $m>0$, then $m_{\ep,\overline{D}_1}\leq m$ for all $|\epsilon|$ sufficiently small.}
\end{itemize}
As a matter of fact, for such $m$, one may choose large coprime integers $p_1,p_2$ such that $q_i=mp_i>q$, for $i=1,2$,. Here, $q>0$ is as in Proposition \ref{P:per-even-mono} and
Remarks \ref{Def-bundlemap}-\ref{R:Def-D1}. Since $m_{0,\overline{D}_1}\leq m$, the stable periods of $F^{q_1}_0|_{\overline{D}_1}$ and $F^{q_2}_0|_{\overline{D}_1}$ are both bounded by $1$. Then, by following the same arguments in the second paragraph of the proof for Proposition \ref{P:bdd-per-Fqeps} (taking $m_1=1$, and $G_\ep=F^{q_1}_\ep$ or $F^{q_2}_\ep$ there), one obtains that the stable periods of $F^{q_1}_\ep|_{\overline{D}_1}$ and $F^{q_2}_\ep|_{\overline{D}_1}$ are both bounded by 1, for any $|\epsilon|$ sufficiently small. This implies that $m_{\ep,\overline{D}_1}\leq (q_1,q_2)=m$, where $(q_1,q_2)$ is the greatest common factor of $q_1$ and $q_2$.

\section{Proofs of Theorem A, Corollary B and Corollary D}\label{S:proof-ThAB-CorD}
In this section, we will utilize our results obtained in previous sections to prove other main results mentioned in the introduction.

\vskip 4mm
\noindent
{\it Proof of Theorem A.} Let $F_{\epsilon}\equiv F_{0}$ in Corollary \ref{P:bdd-per-mono}. Then there exists an integer $m>0$ such that the stable periods of ${F_{0}|}_{A}$ (since $A\subset D_1$) are bounded above by $m$. Thus, together with Proposition \ref{P:da}, we obtain Theorem A immediately. \hfill $\square$

\vskip 4mm
\noindent
{\it Proof of Corollary B.} Let the integer $m>0$ be obtained in Theorem A. Take any integer $i>0$, and let $B_i$ be the open ball centered at the origin with radius $i$. Define
 $$Q_i=\{x\in B_i:\omega(x,F_{0})\text{ is a linearly stable cycle with minimal period at most }m\}.$$
Since $F_{0}$ is compact and pointwise dissipative, the orbit set $O(B_i,F_0)$ is bounded (see \cite[Theorem 2.4.7]{Ha88}). Consequently, together with Theorem A, we can repeat the exactly same arguments in Pol\'{a}\v{c}ik and Tere\v{s}\v{c}\'{a}k \cite[Section 5]{PT92} to obtain that $Q_i$ contains an open dense subset of $B_i$. Now let
$$Q_{0}=\{x\in X:\omega(x,F_{0})\text{ is a linearly stable cycle with minimal period at most }m\}.$$
Note that $X=\mathop{\bigcup}\limits_{i\geq 1}B_i$. Then $Q_0=\mathop{\bigcup}\limits_{i\geq 1}Q_i$ contains an open dense subset of $X$. We have completed the proof.\hfill $\square$

\vskip 4mm
\noindent
{\it Proof of Corollary D.} Let the open bounded subset $D_1$  and the integer $q>0$ be in Theorem \ref{T:im-da-perturb}.
For each $|\epsilon|$ sufficiently small, we define
$$\tilde{Q}_{\epsilon}:=\{x\in D_{1}:\omega(x,F_{\epsilon}^{q})\text{ is a linearly stable cycle with minimal period at most }m_1\},$$
where $m_1$ is from Theorem \ref{T:im-da-perturb}(a). By the $C^1$-robustness for sharpened dynamics alternative in Theorem \ref{T:im-da-perturb}, we can repeat the exactly same arguments in Pol\'{a}\v{c}ik and Tere\v{s}\v{c}\'{a}k \cite[Section 5]{PT92} (with $F$ replaced by $F_\ep^q$ there) to obtain that $\tilde{Q}_{\epsilon}$ contains an open dense subset of $D_{1}$, for any $|\epsilon|$ sufficiently small.

Now, we define
$$Q_{\epsilon}:=\{x\in D_{1}:\omega(x,F_{\epsilon})\text{ is a linearly stable cycle with minimal period at most }m\},$$
where $m=m_1q$. On one hand, it is clear that if $\omega(x,F_{\epsilon})$ is a linearly stable cycle, then $\omega(x,F_{\epsilon}^{q})$ is a linearly stable cycle.
Then, Proposition \ref{P:bdd-per-Fqeps} entails that, $Q_{\epsilon}\subset\tilde{Q}_{\epsilon}$. On the other hand, it follows from the claim in the proof of Corollary \ref{P:bdd-per-mono} that $\tilde{Q}_{\epsilon}\subset Q_{\epsilon}$. Thus, we have proved $Q_{\epsilon}=\tilde{Q}_{\epsilon}$, which completes the proof.\hfill $\square$

\section{An example}\label{S:exam}
To illustrate our abstract results, we present in this section an example of a nonlocal perturbation of a time-periodic parabolic equation. For such non-locally perturbed system, we will show the global dynamics of the improved generic convergence to cycles whose minimal periods are uniformly bounded.
\vskip 2mm

Consider the following nonlocal perturbations of a time-periodic parabolic equation:
\begin{eqnarray}\label{E2}
        \frac{\partial u}{\partial t} &=& \Delta u+f(t,x,u,\nabla u)+\epsilon^2 C(t,x)\int_{\Omega}{p(x)u(t,x)dx}, \quad\ x\in \Omega,\ t>0,\notag\\
       \frac{\partial u}{\partial \nu} &=& 0, \quad\quad\quad\quad\quad\quad\quad\quad\quad\quad\quad\quad\quad\quad\quad\quad\quad\quad\quad\ \ \,x\in \partial\Omega,\ t>0, \\
       u(0,x) &=& u_{0}(x), \quad\quad\quad\quad\quad\quad\quad\quad\quad\quad\quad\quad\quad\quad\quad\quad\quad\ \,\,\,\,x\in\Omega\notag,
\end{eqnarray}
where $\Omega\subset\mathbb{R}^{N}$ ($N\geq1$) is a smooth bounded domain, $\ep\in \mathbb{R}$ a parameter and $\nu$ is the unit outward normal vector field on $\partial\Omega$.
The nonlinearity $f:\mathbb{R}\times\bar{\Omega}\times\mathbb{R}\times{\mathbb{R}^{N}}\to\mathbb{R};(t,x,u,\xi)\mapsto f(t,x,u,\xi)$ is assumed to be of class $C^1$ and $\tau$-periodic in $t$.
The functions $C(t,x),p(x)$ in nonlocal perturbation terms are assumed to be $C^1$-functions, and $C(t,x)$ is $\tau$-periodic in $t$.

Let $Y=L^{p}(\Omega)(N<p<\infty)$. For each $\alpha\in(\frac{1}{2}+\frac{N}{2p},1)$, let $X=Y^{\alpha}$ be the fractional power space associated with $L^p$-realization of $-\Delta$ and the boundary conditions. Then $X\hookrightarrow C^{1+\gamma}(\bar{\Omega})$ with continuous inclusion for $\gamma\in[0,2\alpha-\frac{N}{p}-1)$.
So, $(X,X_+)$ is a strongly ordered Banach space with the solid cone $X_+$ consisting of all nonnegative functions in $X$.

For any $u_0\in X$, equation \eqref{E2} admits a (locally) unique regular solution $u(t,\cdot,u_{0})$ in $X$. Under appropriate growth condition (see e.g., Amann \cite{Amann85}), for any $u_{0}\in X$, \eqref{E2} has a unique global solution $t\mapsto u(t,\cdot,\epsilon,u_{0})$ satisfying $u(0,\cdot,\epsilon,u_{0})=u_{0}(\cdot)$. All the smoothness and compactness required for the period map $F_{\epsilon}: u_{0}\mapsto u(\tau,\cdot,\epsilon,u_{0})$ in (H1)-(H3) are satisfied (see, e.g. \cite{PH93,PT93}).
In particular, for $\ep=0$, $F_{0}$ is pointwise dissipative with an attractor $A$ (see, e.g. \cite{Ha88} or \cite{P02}).

\vskip 2mm
If the $C^1$-functions $C(t,x),p(x)$ in the nonlocal term are nonnegative, then system \emph{\eqref{E2}} can be shown to admit a strong comparison principle (c.f. Hess and Pol\'{a}\v{c}ik \cite[Example 3]{PH93}, or Pol\'{a}\v{c}ik and Tere\v{s}\v{c}\'{a}k \cite[Section 4]{PT93}); and hence, it belongs the class of $C^1$-smooth strongly monotone dynamical systems. Under such circumstance, Corollary B (by taking $F_\ep$ as $F_0$) entails that, for each $\ep$, {\it the dynamics of \emph{\eqref{E2}} remains the generic convergence to cycles with stable periods bounded.}

If $C(t,x),p(x)$ are not nonnegative and $\ep$ is not small, then we note that stable complicated dynamics may occur (see, e.g. Fiedler and Pol\'{a}\v{c}ik \cite{FP91} and references therein).
On the other hand, for all $\ep$ sufficiently small, equation \eqref{E2} can be viewed as an $\ep$-perturbed system. So, by virtue of Corollary D (with $F_\ep$ not necessarily monotone), we can obtain the improved generic convergence for the $\ep$-perturbed system \eqref{E2}, that is,  {\it there is an open neighbourhood $D_{1}$ of $A$ and an integer $m>0$ such that, for any $\ep$ sufficiently small, the set of initial condition $u_0\in D_1$ whose solution $u(t,\cdot,\epsilon,u_{0})$ converges to a linearly stable $k\tau$-periodic solution ($0<k\leq m$), contains an open dense subset of $D_1$.}

Finally, we would like to mention the case that $f$ is independent of $t$. In such situation, for $\ep=0$, $F_0$ can be chosen as the time-$\tau$ map of the unperturbed $C^1$-strongly monotone semiflow generated by $f$. So, all linearly stable periodic points of $F_0$ are fixed points (see \cite[Proposition 9.4]{Hess91}). Due to the fact mentioned at the end of Section \ref{S:bdd-per-Fqeps}, one can find an open neighbourhood $D_{1}$ of $A$ such that, for all $\ep$ sufficiently small, the set of initial condition $u_0\in D_1$ whose solution $u(t,\cdot,\epsilon,u_{0})$ converges to a linearly stable $\tau$-periodic solution, contains an open dense subset of $D_1$.

\end{document}